\documentclass[a4paper,11pt]{amsart}
\usepackage{amssymb}
\usepackage{amscd}
\usepackage{comment}
\usepackage{amsmath,amsthm}
\usepackage[colorlinks=true]{hyperref}
\usepackage{enumerate}
\usepackage{booktabs,multirow}
\usepackage{tikz}
\usepackage{rotating}
\usetikzlibrary{patterns}
\usetikzlibrary{decorations.pathreplacing}
\usetikzlibrary{calc,through}

\allowdisplaybreaks[1]
\numberwithin{figure}{section}
\numberwithin{equation}{section}

\title[Dyck tilings of type $D$]
{Dyck tilings of type $D$}
\author[K.~Shigechi]{Keiichi~Shigechi}
\email{k1.shigechi AT gmail.com}
\date{\today}

\newcommand\tikzpic[2]{
\raisebox{#1\totalheight}{
\begin{tikzpicture}
#2
\end{tikzpicture}
}}
\newtheorem{theorem}[figure]{Theorem}
\newtheorem{example}[figure]{Example}
\newtheorem{lemma}[figure]{Lemma}
\newtheorem{defn}[figure]{Definition}
\newtheorem{prop}[figure]{Proposition}
\newtheorem{cor}[figure]{Corollary}

\newtheorem{remark}[figure]{Remark}
\begin{document}
\begin{abstract}
We introduce and study cover-inclusive and cover-exclusive Dyck tilings of type $D$. 
It is shown that the generating functions of Dyck tilings of type $D$ 
are expressed in terms of the generating function of ballot tilings of 
type $B$.
We introduce link patterns of type $D$ and plane trees for a ballot path,
and construct a map from trees to $\mathbb{Z}[q]$. 
This map gives the generating function of cover-inclusive Dyck tilings of type $D$ 
associated to the ballot path.
\end{abstract}

\maketitle

\section{Introduction}
Cover-exclusive Dyck tiling of type $A$ first appeared in the study 
of Kazhdan--Lusztig polynomials $P^{-}_{\lambda,\mu}$, where 
$\lambda$ and $\mu$ are Dyck paths,  
for the maximal parabolic subgroup \cite{Bre02}.
As shown in \cite{Deo87}, we have two natural modules $\mathcal{M}^{\pm}$ 
to study (maximal) parabolic Kazhdan--Lusztig polynomials. 
The polynomial $P^{\pm}_{\lambda,\mu}$ is associated to 
the module $\mathcal{M}^{\pm}$.  
By taking the ``inverse" of the matrix $P^{-}_{\lambda,\mu}$, 
Zinn-Justin and the author introduced the notion of cover-inclusive 
Dyck tilings to study Kazhdan--Lusztig polynomials $P^{+}_{\lambda,\mu}$ \cite{SZJ12}.  
Independently, Kenyon and Wilson also introduced cover-inclusive 
Dyck tilings in the study of double dimer models \cite{KW11,KW15} 
(conjectures in \cite{KW11} are proved in \cite{K12}).
Since then, cover-inclusive Dyck tilings appear in the connection 
to other research fields such as
Schramm--Loewner evolution \cite{KKP17,PelWu19,Pon18}, fully packed loops \cite{FN12}, and 
intersection cohomology of Grassmannian Schubert varieties \cite{Pa19}.

There are several variants of Dyck tilings and they are studied
in \cite{JVK16,KW11,K12,KMPW12,S17,S19,S20}.
The first one is symmetric Dyck tilings studied in \cite{JVK16,S20}.
The second one is ballot tiling of type $B$ studied in \cite{S17}.
We simply say a ballot tiling  instead of a ballot tiling of type $B$,
since ballot paths are already a type $B$ analogue of Dyck paths.
The third one is $k$-Dyck tilings \cite{JVK16}. A generalization of Dyck paths 
is $k$-Dyck paths and they contain rich combinatorial properties 
on its own.
In this paper, we introduce Dyck tilings of type $D$ and study 
its relation to ballot tilings.

Since the modules $\mathcal{M}^{\pm}$ are naturally identified with 
the Hermite symmetric space in case of the maximal parabolic groups, 
one can consider Dyck tilings according to the types of Hermite symmetric
space \cite{Boe88,S141}.
We consider the following three Hermite symmetric space \cite{Boe88,LejStr13,S141}:
\begin{align*}
(A_{n},A_{k}\times A_{n-k}), \quad
(B_{n},A_{n-1}),\quad
(D_{n},A_{n-1}).
\end{align*}
We call these spaces type $A$, $B$ and $D$ from left to right.

Another view for Dyck tilings is to associate Dyck tilings to Temperley--Lieb 
algebras studied in mathematical physics.
Temperley--Lieb algebras can be regarded as Hecke algebras with quotient 
relations.
Then, Kazhdan--Lusztig polynomials $P^{-}_{\lambda,\mu}$ are easily 
computed by use of the notion of link patterns.
A link pattern is a perfect matching without crosses of links. 
In this context, one can naturally consider variants of Temperley--Lieb 
algebra according to the classification of the Coxeter groups.
Note that when we say Dyck tilings of type $A$, $D$ or ballot tilings of type $B$, 
we always keep in mind this classification.
Note that this classification coincides with the classification 
by Hermite symmetric spaces.

Dyck tilings are tilings in the region surrounded by two paths $\lambda\le\mu$.
We give a weight to a tile forming a tiling and the generating function
with fixed paths is defined as the sum of weights of configurations of tiles.
We have natural three statistics $\mathrm{area}(D)$, $\mathrm{tiles}(D)$ and 
$\mathrm{art}(D)$ on a tiling $D$.
The statistic $\mathrm{tiles}(D)$ counts the number of tiles forming a tiling, and 
$\mathrm{area}(D)$ counts the number of boxes forming a tile in a tiling.
Then, the statistics $\mathrm{art}(D)$ is one-half of the sum of $\mathrm{tiles}(D)$
and $\mathrm{area}(D)$.  
Given a cover-inclusive tiling, we have Kazhdan--Lusztig polynomials by statistics $\mathrm{tiles}$ \cite{LS81,SZJ12},
and the generating function in terms of $q$-integers by $\mathrm{art}$ \cite{KW11,K12,KMPW12}. 
In general, the former cannot be written in terms of fractions of $q$-integers.
However, the latter, which is written as a product of fractions of $q$-integers, takes 
a value in $\mathbb{N}[q]$ \cite{K12,S17,S20}. 

The two classes of tilings, cover-inclusive and cover-exclusive tilings, 
are related by the inverse of the incidence matrix.
The entries of the incidence matrix encode information about cover-exclusive tilings.
Then, the incidence matrix and its inverse represent a relation of two parabolic 
Kazhdan--Lusztig polynomials $P^{\pm}_{\lambda,\mu}$ as observed in \cite{Deo87}.
In terms of the incidence matrix, this duality between cover-inclusive 
and cover-exclusive tiligns comes from the Principle of Inclusion-Exclusion 
(see for example \cite[Chapter 2]{Sta97}).
The sizes of tiles in a cover-inclusive tiling are weakly decreasing from 
bottom to top, whereas the ones in a cover-exclusive tiling are strictly
increasing from bottom to top.

Let $\lambda:=\lambda_1\ldots\lambda_{n}$ be a path of length $n$ and 
$\widetilde{\lambda}:=\lambda_1\ldots\lambda_{n-1}$ be a path of length $n-1$. 
The main result of this paper is that the generating function of Dyck tilings 
of type $D$ associated to a path $\lambda$ is equal to the generating function 
of ballot tilings associated to a path $\widetilde{\lambda}$.
We introduce the notions of a link pattern $\pi$ of type $D$ and 
of a plane tree associated to $\pi$ for a path $\lambda$. 
Combining these combinatorial objects, we give a map from trees to 
$\mathbb{N}[q]$ which yields the generating function of cover-inclusive 
Dyck tilings of type $D$.

Once we have a relation between Dyck tilings of type $D$ and ballot tilings 
of type $B$, the results for ballot tilings in \cite{S17,S20} can be naturally 
applied to the type $D$ case. 

The paper is organized as follows.
In Section \ref{sec:DTABD}, we introduce Dyck and ballot tiles to define 
cover-inclusive and cover-exclusive Dyck tilings of type $D$.
In Section \ref{sec:IM}, we study the incidence matrices. 
We show that an entry of the matrices is a generating function 
of Dyck tilings of type $D$ calculated with the weights $\mathrm{tiles}$
or $\mathrm{art}$ on Dyck or ballot tiles.
In Section \ref{sec:EnuDTD}, we study enumerations of Dyck tilings of 
type $D$. We show that the generating function of type $D$ is expressed 
in terms of the generating function of ballot tiles of type $B$.
We give a link pattern of type $D$ and a plane tree for a path $\lambda$.
We also construct a map from trees to $\mathbb{Z}[q]$ to obtain 
the generating function with the weight $\mathrm{art}$.

\paragraph{\bf Notations}
We introduce the quantum integer $[n]:=\sum_{i=0}^{n-1}q^{i}$,  
quantum factorial $[n]!:=\prod_{i=1}^{n}[i]$, $[2m]!!:=\prod_{i=1}^{m}[2i]$, 
and the $q$-analogue of the binomial coefficients
\begin{align*}
\genfrac{[}{]}{0pt}{}{n}{m}:=\frac{[n]!}{[n-m]![m]!}, 
\qquad
\genfrac{[}{]}{0pt}{}{n}{m}_{q^{2}}:=\frac{[2n]!!}{[2(n-m)]!!\cdot[2m]!!}.
\end{align*}

\section{Dyck tilings of type \texorpdfstring{$ABD$}{ABD}}
\label{sec:DTABD}
In this section, we define Dyck tilings of type $A$, $D$ and ballot 
tilings of type $B$.
For this purpose, we introduce Dyck and ballot tiles which form
a tiling.
Further, we have two classes of Dyck tilings of any types, which 
have constraint on a configuration of Dyck and ballot tiles,
namely, they are cover-inclusive and cover-exclusive tilings. 
Thus, in this paper, we consider Dyck/ballot tilings of six species.

A {\it Dyck path of size $n$} is a lattice path from the origin 
$(0,0)$ to $(2n,0)$ with up (``U" or $(1,1)$)  steps and down 
(``D" or $(1,-1)$) steps, which does not go below the horizontal 
line $y=0$.
We call a sequence of $U$ and $D$ corresponding to a Dyck path 
a Dyck word.

A {\it ballot path of size $(n,n')$} is a lattice path from the 
origin $(0,0)$ to $(2n,n')$ with up and down steps
which does not go below the horizontal line $y=0$. 
A ballot path also consists of up steps and down steps.
Note that a Dyck path of size $n$ is a ballot path of size $(n,0)$.

Let $\epsilon:=0$ or $1$.
A {\it path of type $D$ with sign $\epsilon$} is a lattice path from the origin 
$(0,0)$ to $(n,n')$ with $n'\in\mathbb{Z}$ and $n'\equiv n+2\epsilon\pmod4$. 
A path of type $D$ with sign $\epsilon$ consists of $U$ steps and $D$ steps.
The length of a bollot path is $n$.
When $n$ even and $\epsilon=0$ (resp. $1$), 
the highest path is $U\ldots U$ (resp. $U\ldots UD$) and 
the lowest path is $D\ldots D$ (resp. $D\ldots DU$).
Similarly, when $n$ odd and $\epsilon=0$ (resp. $1$),
the highest path is $U\ldots U$ (resp. $U\ldots UD$)
and the lowest path is $D\ldots DU$ (resp. $D\ldots D$).
 
The highest and lowest paths of type $D$ and a line $x=n$ determine 
a region $R$ in a plane.
Given $n$ and $\epsilon$, we put two-by-two tiles rotated in 45 degree whose 
centers are in $(n,m)$ with $m\equiv n+2(\epsilon+1)\pmod4$.
In the remaining region, we put a unit tile rotated 
in 45 degree (see Figure \ref{fig:Region}).
Once we have fixed the top and lowest paths as above, 
other paths are obtained as lattice paths in the 
region $R$. 
In case of the middle picture in Figure \ref{fig:Region},
we have eight admissible paths.
They are $UUUD$, $UUDU$, $UDUU$, $DUUU$, $UDDD$, $DUDD$,
$DDUD$, and $DDDU$.

Let $\mathcal{P}_{n}\in\{U,D\}^{n}$ be a path of length $n$,
and $\mathcal{P}_{n}^{\epsilon}\subset \mathcal{P}_{n}$,
$\epsilon\in\{0,1\}$, be 
a path of type $D$ with sign $\epsilon$.
Then, we have 
\begin{align}
\label{eqn:setpath}
\mathcal{P}_{n}=\bigoplus_{\epsilon\in\{0,1\}}\mathcal{P}_{n}^{\epsilon}.
\end{align} 
Note that given a path of type $D$ of length $n$, the 
sign $\epsilon$ is either $0$ or $1$ from Eqn. (\ref{eqn:setpath}).

\begin{remark}
We have a ballot path by concatenating $U^{m}$ to a path 
of type $D$ for some $m\ge0$.
\end{remark}

In a two-by-two tile, the north and south one-by-one tiles 
are called {\it anchor boxes} (see the right picture in 
Figure \ref{fig:Region}).
\begin{figure}[ht]
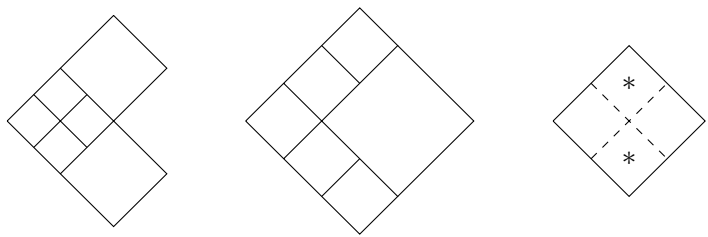

\begin{eqnarray*}
\tikzpic{-0.5}{[x=0.7cm,y=0.7cm]
\draw(0,0)--(2,2)(2,2)--(3,1)--(1,-1)(1,1)--(3,-1)--(2,-2)--(0,0);
\draw(0.5,0.5)--(1.5,-0.5)(0.5,-0.5)--(1.5,0.5);
}\qquad
\tikzpic{-0.5}{
\draw(0,0)--(1.5,1.5)--(2,1)--(0.5,-0.5)(0,0)--(1.5,-1.5)--(2,-1)--(0.5,0.5);
\draw(1,1)--(1.5,0.5)(1,-1)--(1.5,-0.5);
\draw(2,1)--(3,0)--(2,-1);
}\qquad
\tikzpic{-0.5}{
\draw(0,0)--(1,1)--(2,0)--(1,-1)--(0,0);
\draw[dashed](0.5,0.5)--(1.5,-0.5)(0.5,-0.5)--(1.5,0.5);
\draw(1,0.5)node{$\ast$}(1,-0.5)node{$\ast$};
}
\end{eqnarray*}
\caption{A tiling in $R$ for $L=4$. 
$\epsilon=0$ (left picture) and $\epsilon=1$ (middle picture).
Anchor boxes (marked by $\ast$) in a two-by-two tile (right picture).
}
\label{fig:Region}
\end{figure}

Let $\lambda$ and $\mu$ be two lattice paths defined as above.
If a path $\mu$ is above $\lambda$, we denote it by $\lambda\le\mu$. 

A {\it Dyck tile} is a ribbon (a skew shape which does not contain a two-by-two box)
such that the centers of the boxes form a Dyck  path.
Similarly, a {\it ballot tile of type $B$} is a ribbon such that 
the centers of the boxes form a ballot path.
By definition, a Dyck tile is a ballot tile, but the reverse is not true.
We will define a Dyck or a ballot tile of type $D$ by gluing 
a Dyck tile with two-by-two tile, or by gluing two ballot tiles of type $B$ and a
single box into a single tile.

Let $\lambda$ and $\mu$ be a path and $\lambda\le\mu$.
A tiling is said to be {\it cover-inclusive}
if the tiling satisfies the following conditions: 
\begin{enumerate}
\item
If we translate a Dyck/ballot tile $d_1$ without two-by-two boxes downward 
 by $(0,-2)$, 
then it is strictly below the path $\lambda$ or contained 
in another Dyck/ballot tile $d_2$.
Here, the size of $d_2$ is equal to or larger than the one of $d_1$ if cover-inclusive.
\item If we translate a Dyck/ballot tile $d_1$ with two-by-two boxes downward 
 by $(0,-4)$, 
then it is strictly below the path $\lambda$  or contained 
in another Dyck/ballot tile $d_2$.
Here, the size of $d_2$ is equal to or larger than the one of $d_1$.
\item We regard the size of a two-by-two box as one.
\end{enumerate}

Similarly, a tiling is said to be {\it cover-exclusive}
if the tiling satisfies the following conditions.
Let $d_1$ and $d_2$ be Dyck/ballot tiles and $d_1$ is above $d_2$.
\begin{enumerate}
\item
If there exists a box of $d_1$ just above, north-west of north-east of $d_2$,
then all boxes just above, north-west and north-east of a box of $d_2$
belong to $d_1$ or $d_2$.
\item
If $d_2$ contains a two-by-two box, then $d_1$ should also contain 
a two-by-two box.
\end{enumerate}

Roughly speaking, the sizes of tiles are weakly decreasing 
from the bottom to top in case of  a cover-inclusive tiling, whereas 
the sizes of tiles are strictly increasing in case of a cover-exclusive
tiling.
This is due to the Principle of Inclusion-Exclusion 
({\it e.g.} see Chapter 2 in \cite{Sta97}).

We introduce three types of tilings: 1) a Dyck tiling of type $A$,
2) a ballot tiling of type $B$, and 3) a Dyck tiling of type $D$.

\begin{defn}
Let $\lambda\le\mu$ be two Dyck paths and $R_{A}(\lambda,\mu)$ be a region 
surrounded by $\lambda$ and $\mu$. 
A Dyck tiling of type $A$ is a tiling by Dyck tiles in the region $R_{A}(\lambda,\mu)$. 
We have two types of Dyck tilings of type $A$: a cover-inclusive Dyck tiling 
and a cover-exclusive Dyck tiiing.
\end{defn}

Let $\lambda$ be a ballot path of length $(n,n')$, and $\mu\ge\lambda$ be 
a ballot path.
We denote by $R_{B}(\lambda,\mu)$ the region surrounded by $\lambda$, $\mu$
and the line $x=2n+n'$.
We put $\ast$ on the box at the line $x=2n+n'$ and call them {\it anchor boxes}.
\begin{defn}
\label{defn:bt}
A ballot tiling (of type $B$) is a tiling in the region $R_{B}(\lambda,\mu)$ by Dyck and ballot tiles.
A ballot tiling satisfies the following constraints:
\begin{enumerate}
\item
A right-most box of a ballot tile is on an anchor box.
\item
The number of ballot tiles of length $(n,n')$ is even for $n'\in2\mathbb{N}_{\ge0}+1$
and zero for $n'\in2\mathbb{N}_{\ge1}$.
\item
The right-most box of a Dyck tile of size $n\ge1$ may be on an anchor box.
\end{enumerate}
We have two types of ballot tilings: cover-inclusive ballot tilings
and cover-exclusive ballot tilings.
\end{defn}

Fix two paths $\lambda,\mu$ of type $D$ with sign $\epsilon$, where $\lambda$ is below
$\mu$.
Then, paths $\lambda, \mu$ and the line $x=n$ determine a region $R_{D}(\lambda,\mu)$ 
with one-by-one tiles and two-by-two tiles.
We put Dyck tiles and ballot tiles of type $B$ in the region $R_{D}(\lambda,\mu)$.

A {\it Dyck tiling of type $B$} is a tiling of the region $R_{D}(\lambda,\mu)$ 
by Dyck tiles and ballot tiles of type $B$ with the cover-inclusive 
or cover-exclusive property.
Here, the right-most box of a ballot tile of type $B$ is on an anchor box 
in a two-by-two box.

\begin{remark}
Note that a Dyck tiling of type $B$ is different from a ballot tiling of type $B$.
A ballot tiling of type $B$ can be regarded as a Dyck tiling of type $B$ in the region
$R_{B}(\lambda,\mu)$ with some conditions on ballot tiles as in Definition \ref{defn:bt}.
Similarly, we realize a Dyck tiling  of type $D$ as a Dyck tiling of type $B$ with some conditions
on ballot tiles.
\end{remark}

\begin{defn}
\label{defn:DTD}
A Dyck tiling of type $D$ is a Dyck tilings of type $B$ with the following 
constraints:
\begin{enumerate}
\item the right most box of a ballot tile of length $(n,n')$ with $n'\ge0$ 
is on an anchor box in a two-by-two tile.
\item The number of ballot tiles of length $(n,n')$ is even for $n'\in2\mathbb{N}_{\ge0}$ 
and zero for $n'\in2\mathbb{N}_{\ge0}+1$.
\item If a Dyck tile contains a west box of a two-by-two box but not an anchor box,
then the two-by-two box and the Dyck tile form a Dyck tile of type $D$.
\end{enumerate}
We have two types of Dyck tilings of type $D$: cover-inclusive and cover-exclusive.
\end{defn}

From conditions (1) and (2) in Definition \ref{defn:DTD}, we have two ballot 
tiles of type $B$ which have anchor boxes 
in the same two-by-two tile.
Then, there exists a unique box $b$ such that the northeast and southeast edges 
are shared by the two ballot tiles.
We glue together the two ballot tiles and the box $b$ as a single tile.
We call this glued tiles a ballot tile of type $D$.
\begin{remark}
The condition (1) in Definition \ref{defn:DTD} does not mean that all Dyck 
tiles have an anchor box in
a two-by-two tile.
The condition (1) is applied when a Dyck tile is regarded as a ballot tile.
Thus, we may have a non-trivial Dyck tile whose right-most box is not on
an anchor box.
The condition (2) in Definition \ref{defn:DTD} should be understood under the 
condition (1).
\end{remark}

Note that the condition (2) in Definition \ref{defn:bt} is different from 
the condition (2) in Definition \ref{defn:DTD}.
This difference plays a central role when we prove the correspondence between 
ballot tilings of type $B$ and Dyck tilings of type $D$ 
in Theorem \ref{thrm:DB} and Theorem \ref{thrm:DB2}.

We define statistics $\mathrm{area}$, $\mathrm{tiles}$ 
and $\mathrm{art}$ on Dyck tiles, ballot tiles, and ballot tiles of type $D$.

\paragraph{\bf Dyck tilings of type A}
All tiles forming a Dyck tiling of type A are Dyck tiles.
Recall that a Dyck tile is made from single boxes. 
For a Dyck tile $D$, we define $\mathrm{area}(D)$ as the number of boxes
which form the Dyck tile $D$ and the number of tiles as $\mathrm{tiles}(D):=1$.

\paragraph{\bf Ballot tlings}
A tile forming a ballot tiling is either a Dyck or ballot tile.
Given a Dyck tile $D$, we define $\mathrm{area}(D)$ and $\mathrm{tildes}(D)$
in a similar way as Dyck tilings of type A.
Let $B$ be a ballot tile of length $(n,n')$.
Then, the statistics $\mathrm{area}(B)$ is defined one-half plus the number 
of boxes without $\ast$ which form $B$, and the statistics $\mathrm{tiles}(B)$
is one-half.
Since the number of ballot tiles is always even by Definition \ref{defn:bt},
we regard two ballot tiles of the same shape and a unique single box which 
is next to both two ballot tiles as a single tile.

\paragraph{\bf Dyck tilings of type $D$}
We have three types of tiles: 1) a Dyck tile, 2) a Dyck tile of type $D$,
and 3) a ballot tile of type $D$.
The statistics on a Dyck tile is the same as the case of Dyck tilings of type $A$.
We define statistics $\mathrm{area}$ and $\mathrm{tiles}$ 
for a Dyck and a ballot tile of type $D$:
\begin{enumerate}
\item
Note that when a Dyck tile $D$ includes a box in two-by-two tile, $\mathrm{area}$ of 
the two-by-two tile is one. 
In other words, $\mathrm{area}$ is the number of boxes forming $D$ and their sizes are 
irrelevant. 
Recall that a Dyck tile $\mathcal{D}$ of type $D$ consists of one-by-one boxes and a single 
two-by-two box. Then, we define $\mathrm{tiles}$ of $\mathcal{D}$ is one.
\item
For a ballot tile $B$ of type $B$, we define $\mathrm{area}(B)$ is the 
number of one-by-one boxes which is not an anchor box.
Recall that a ballot tile $\mathcal{B}$ of type $D$ consists of two ballot tiles of type $B$ 
and a single box.  
Then, we define the statistics $\mathrm{tiles}(\mathcal{B}):=1$.
\end{enumerate}

Finally, for a tile $\mathcal{D}$ of any type, 
we define a statistics $\mathrm{art}$ as 
\begin{eqnarray*}
\mathrm{art}(\mathcal{D}):=(\mathrm{area}(\mathcal{D})+\mathrm{tiles}(\mathcal{D}))/2.
\end{eqnarray*}
\begin{remark}
In the definitions of $\mathrm{area}$ and $\mathrm{tiles}$ for a ballot tile in a
ballot tiling, they are not an integer. 
However, the statistics $\mathrm{art}$ on a ballot tile is always a positive integer.
\end{remark}

Examples of Dyck and ballot tiles of type $D$ are in Figure \ref{fig:Dycktile}.
\begin{figure}[ht]
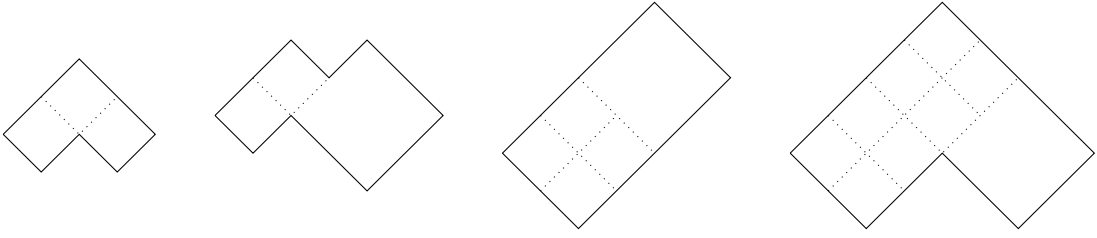

\tikzpic{-0.5}{
\draw(0,0)--(1,1)--(2,0)--(1.5,-0.5)--(1,0)--(0.5,-0.5)--(0,0);
\draw[dotted](0.5,0.5)--(1,0)--(1.5,0.5);
}\quad
\tikzpic{-0.5}{
\draw(0,0)--(1,1)--(1.5,0.5)--(2,1)--(3,0)--(2,-1)--(1,0)--(0.5,-0.5)--(0,0);
\draw[dotted](0.5,0.5)--(1,0)--(1.5,0.5);
}\quad
\tikzpic{-0.5}{
\draw(0,0)--(2,2)--(3,1)--(1,-1)--(0,0);
\draw[dotted](1,1)--(2,0)(0.5,0.5)--(1.5,-0.5)(0.5,-0.5)--(1.5,0.5);
}\quad
\tikzpic{-0.5}{
\draw(0,0)--(2,2)--(4,0)--(3,-1)--(2,0)--(1,-1)--(0,0);
\draw[dotted](1,1)--(2,0)--(3,1)(0.5,0.5)--(1.5,-0.5)(1.5,1.5)--(2.5,0.5);
\draw[dotted](0.5,-0.5)--(2.5,1.5);
}
\caption{Dyck tiles and ballot tiles of type $D$. The statistics $\mathrm{art}$ are 
$2,2,3$ and $5$ from left to right.}
\label{fig:Dycktile}
\end{figure}

Let $\lambda, \mu$ be two paths such that $\lambda$ is below $\mu$.
We denote by $\mathcal{D}_{X}(\lambda/\mu)$ the set of 
Dyck/ballot tilings of type $X$ in the region $R_{X}(\lambda,\mu)$.
Here, $X$ is a type of a Dyck/ballot tiling, so $X=A,B$ or $D$.
We define $\mathcal{D}_{X}^{I}(\lambda/\mu)$ (resp. $\mathcal{D}_{X}^{II}(\lambda/\mu)$)
be the set of cover-inclusive (resp. cover-exclusive) tilings of type $X$.

We define 
\begin{eqnarray*}
\mathcal{D}^{I}_{X}(\lambda/\ast)&:=&\bigcup_{\mu}\mathcal{D}^{I}_{X}(\lambda/\mu), \\
\mathcal{D}^{II}_{X}(\ast/\mu)&:=&\bigcup_{\lambda}\mathcal{D}^{II}_{X}(\lambda/\mu).
\end{eqnarray*}

The generating functions for tilings are defined by a sum of the weights given 
to tiles. 
\begin{defn}
Let $\lambda$ be a path. 
We define generating functions $P_{\lambda,\mu}$, $P^{X}_{\lambda}$ and 
$\widetilde{P}^{X}_{\mu}$ as 
\begin{align}
\label{eqn:dtbt}
P^{X}_{\lambda,\mu}&:=\sum_{D\in\mathcal{D}^{I}_{X}(\lambda/\mu)}q^{\mathrm{art}(D)},\\
\label{eqn:dtb}
P_{\lambda}^{X}&:=\sum_{D\in\mathcal{D}^{I}_{X}(\lambda/\ast)}q^{\mathrm{art}(D)}, \\ 
\widetilde{Q}^{X}_{\lambda,\mu}&:=\sum_{D\in\mathcal{D}^{I}_{X}(\lambda/\mu)}q^{\mathrm{tiles}(D)},\\
\label{eqn:dtt}
\widetilde{P}^{X}_{\mu}&:=\sum_{D\in\mathcal{D}^{II}_{X}(\ast/\mu)}q^{\mathrm{tiles}(D)}, 
\end{align}
where $X\in\{A,B,D\}$ is a type of Dyck tilings. 
\end{defn}

\begin{example}
Let $\lambda=DDUUDD$. Then, we have $36$ configurations 
in $\mathcal{D}^{I}_{D}(\lambda/\ast)$.
We have six configurations satisfying $\mathrm{art}(D)=5$.
These six configurations are listed in Figure \ref{fig:config}.
\begin{figure}[ht]
\scalebox{0.7}{
\tikzpic{-0.5}{
\draw(0,0)--(1,1)--(2,0)--(1,-1)--(0,0)(2,0)--(3,1)--(4,0)--(3,-1)--(2,0);
\draw(0.5,0.5)--(1.5,-0.5)(0.5,-0.5)--(1.5,0.5);
\draw[dotted](1,1)--(3,3)--(4,2)--(3,1)(2,2)--(3,1)(1.5,1.5)--(2.5,0.5)
(1.5,0.5)--(2.5,1.5);
}\qquad
\tikzpic{-0.5}{
\draw(0,0)--(0.5,0.5)(1.5,0.5)--(2,0)--(1,-1)--(0,0)(2,0)--(3,1)--(4,0)--(3,-1)--(2,0);
\draw(0.5,0.5)--(1.5,-0.5)(0.5,-0.5)--(1.5,0.5);
\draw[dotted](0.5,0.5)--(3,3)--(4,2)--(3,1)(2,2)--(3,1)(1,1)--(1.5,0.5)--(2.5,1.5);
\draw[dotted](1.5,1.5)--(2,1);
\draw(1.5,0.5)--(2,1)--(2.5,0.5);
}\qquad
\tikzpic{-0.5}{
\draw[dotted](0,0)--(3,3)--(4,2)--(3,1);
\draw(0.5,-0.5)--(2.5,1.5)--(4,0)--(3,-1)--(2,0)
(0.5,-0.5)--(1,-1)--(3,1)(1,0)--(1.5,-0.5)(1.5,0.5)--(2,0)(2,1)--(2.5,0.5);
\draw[dotted](0,0)--(0.5,-0.5)(0.5,0.5)--(1,0)(1,1)--(1.5,0.5)(1.5,1.5)--(2,1)
(2,2)--(2.5,1.5);
}}\\
\scalebox{0.7}{
\tikzpic{-0.5}{
\draw(0,0)--(1,1)--(1.5,0.5)--(2,1)--(2.5,0.5)--(3,1)--(4,0)--(3,-1)--(2,0)
(2,0)--(1,-1)--(0,0);
\draw(0.5,0.5)--(1.5,-0.5)(0.5,-0.5)--(1.5,0.5);
\draw[dotted](1,1)--(3,3)--(4,2)--(3,1)(2,2)--(3,1)
(1.5,1.5)--(2,1)--(2.5,1.5)(1.5,0.5)--(2,0)--(2.5,0.5);
}\qquad
\tikzpic{-0.5}{
\draw(0,0)--(0.5,0.5)--(1,0)--(2.5,1.5)--(4,0)--(3,-1)--(2,0)
(2,0)--(1,-1)--(0,0)(0.5,-0.5)--(1,0)--(1.5,-0.5)(2,1)--(2.5,0.5)--(3,1);
\draw[dotted](0.5,0.5)--(3,3)--(4,2)--(3,1)(1,1)--(2,0)--(2.5,0.5)
(1.5,1.5)--(2,1)(2,2)--(2.5,1.5);
}\qquad
\tikzpic{-0.5}{
\draw(0,0)--(2,2)--(4,0)--(3,-1)--(2,0)--(1,-1)--(0,0);
\draw[dotted](2,2)--(3,3)--(4,2)--(3,1)(0.5,-0.5)--(2.5,1.5)
(2,0)--(3,1)(0.5,0.5)--(1.5,-0.5)(1,1)--(2,0)(1.5,1.5)--(2.5,0.5);
}
}
\caption{Six configurations with $\mathrm{art}(D)=5$.}
\label{fig:config}
\end{figure}	
\end{example}

\begin{remark}
Dyck tilings of type $A$ are treated in \cite{SZJ12}.
Similarly, ballot tilings of type $B$ are in \cite[Section 3]{S17}.
In this paper, ballot tilings of type $B$ means that the ballot tilings
of type $BI$ in \cite{S17}.
In any types, generating functions are defined in a similar way except 
the definitions of admissible Dyck or ballot tiles and their statistics. 
\end{remark}

\section{Incidence matrix}
\label{sec:IM}
We consider the following operations on a path of type $D$:
\begin{enumerate}[($\heartsuit1$)]
\item If $U$ and $D$ are next to each other in this order, 
we make a pair of $U$ and $D$ by connecting them via a simple arc.
\item We continue the above process until the remaining sequence 
is $D\dots DU\dots U$. 
\item We make a pair of two $U$'s, which are next to each other, 
from the right end of $U\ldots U$. 
We connect the two $U$'s via a dashed arc.
\end{enumerate}

Suppose that $i$-th entry and $j$-th entry, $i>j$, in $w$ form
an arc or a dashed arc. 
Then, the size of the (possibly dashed) arc is given by
$|i-j+1|/2$.

We define two operations, $UD$-flipping and $UU$-flipping on 
a path $\lambda$.
The $UD$-flipping is an operation to change a pair of $U$ and 
$D$ into $D$ and $U$.
Similarly, $UU$-flipping is an operation to change a pair of 
two $U$'s into two $D$'s.

We introduce a relation $\leftarrow$ on paths of type $D$.
Let $\lambda_1$ and $\lambda_2$ be paths of type $D$.
We say $\lambda_{1}\leftarrow\lambda_{2}$ if $\lambda_{1}$ 
can be obtained from $\lambda_{2}$ by $UD$-flippings of the 
paired $UD$ connected by a simple arc associated with $\lambda_{2}$, 
and $UU$-flippings of the paired $UU$ connected by a dashed arc.

Suppose that $\lambda_{1}\leftarrow\lambda_{2}$ by a single $UD$-flipping.
Let $m$ be the size of the $UD$-flipped simple arc.
Then, we define the weight as 
$\mathrm{wt}^{I}(\lambda_{1}\leftarrow\lambda_{2})=-q^{m}$, 
and $\mathrm{wt}^{II}(\lambda_{1}\leftarrow\lambda_{2})=-q$.

Suppose that $\lambda_{1}\leftarrow\lambda_{2}$ by a single $UU$-flipping.
Let $m$ be the size of the $UU$-flipped dashed arc and the position  of the right 
$U$ in the $UU$ pair is $r$ from the right end.
Then, we define the weight as 
$\mathrm{wt}(\lambda_{1}\leftarrow\lambda_{2})=-q^{m+r-1}$ and 
$\mathrm{wt}^{I}(\lambda_{1}\leftarrow\lambda_{2})=-q$.

When $\lambda_1\leftarrow\lambda_2$, the weight 
$\mathrm{wt}^{\alpha}(\lambda_1\leftarrow\lambda_2)$, $\alpha\in\{I,II\}$ is the products of weights of 
admissible flippings.

\begin{defn}
The incidence matrix $M:=M_{\lambda,\mu}$ and $N:=N_{\lambda,\mu}$ are defined as 
\begin{eqnarray*}
M_{\lambda,\mu}&:=\mathrm{wt}^{I}(\lambda\leftarrow\mu)\cdot\delta_{\{\lambda\leftarrow\mu\}}, \\
N_{\lambda,\mu}&:=\mathrm{wt}^{II}(\lambda\leftarrow\mu)\cdot\delta_{\{\lambda\leftarrow\mu\}},
\end{eqnarray*}
where $\delta_{S}=1$ in case of $S$ true and $\delta_{S}=0$ otherwise.
\end{defn}

\begin{example}
\label{ex:IM}
The incidence matrix $M$ and $N$ for a path of type $D$ of length $4$
with sign $\epsilon=0$ are as follows. 
The order of bases is $UUUU, UUDD, UDUD, UDDU, DUUD, DUDU, DDUU$ and 
$DDDD$.
\begin{align*}
M&=\begin{pmatrix}
1 & 0 & 0 & 0 & 0 & 0 & 0 & 0\\ 
-q & 1  & 0 & 0 & 0 & 0 & 0 & 0 \\
0 & -q & 1 & 0 & 0 & 0 & 0 & 0 \\
0 & 0 & -q & 1 & 0 & 0 & 0 & 0 \\
0 & 0 & -q & 0 & 1 & 0 & 0 & 0 \\
0 & -q^{2} & q^{2} & -q & -q & 1 & 0 & 0 \\
-q^{3} & q^{3} & 0 & 0 & 0 & -q & 1 & 0 \\
q^{4} & 0 & 0 & 0 & 0 & 0 & -q &  1 
\end{pmatrix}, \\
M^{-1}&=\begin{pmatrix}
1 & 0 & 0 & 0 & 0 & 0 & 0 & 0\\ 
q & 1  & 0 & 0 & 0 & 0 & 0 & 0 \\
q^{2} & q & 1 & 0 & 0 & 0 & 0 & 0 \\
q^{3} & q^{2} & q & 1 & 0 & 0 & 0 & 0 \\
q^{3} & q^{2} & q & 0 & 1 & 0 & 0 & 0 \\
q^{3}+q^{4} & q^{3}+q^{2} & q^{2} & q & q & 1 & 0 & 0 \\
q^{3}+q^{5} & q^{4} & q^{3} & q^{2} & q^{2} & q & 1 & 0 \\
q^{6} & q^{5} & q^{4} & q^{3} & q^{3} & q^{2} & q &  1 
\end{pmatrix}, \\
N&=\begin{pmatrix}
1 & 0 & 0 & 0 & 0 & 0 & 0 & 0\\ 
-q & 1  & 0 & 0 & 0 & 0 & 0 & 0 \\
0 & -q & 1 & 0 & 0 & 0 & 0 & 0 \\
0 & 0 & -q & 1 & 0 & 0 & 0 & 0 \\
0 & 0 & -q & 0 & 1 & 0 & 0 & 0 \\
0 & -q & q^{2} & -q & -q & 1 & 0 & 0 \\
-q & q^{2} & 0 & 0 & 0 & -q & 1 & 0 \\
q^{2} & 0 & 0 & 0 & 0 & 0 & -q &  1 
\end{pmatrix}, \\
N^{-1}&=\begin{pmatrix}
1 & 0 & 0 & 0 & 0 & 0 & 0 & 0\\ 
q & 1  & 0 & 0 & 0 & 0 & 0 & 0 \\
q^{2} & q & 1 & 0 & 0 & 0 & 0 & 0 \\
q^{3} & q^{2} & q & 1 & 0 & 0 & 0 & 0 \\
q^{3} & q^{2} & q & 0 & 1 & 0 & 0 & 0 \\
q^{2}+q^{4} & q+q^{3} & q^{2} & q & q & 1 & 0 & 0 \\
q+q^{5} & q^{4} & q^{3} & q^{2} & q^{2} & q & 1 & 0 \\
q^{6} & q^{5} & q^{4} & q^{3} & q^{3} & q^{2} & q &  1 
\end{pmatrix}.
\end{align*}
\end{example}

Let $d_1$ and $d_2$ be Dyck/ballot tiles in a cover-exclusive Dyck tiling of type $D$, 
and $d_1$ is left to $d_2$.
Then, the distance between the center of the right-most box in $d_1$ and 
the one of the left-most box in $d_2$ is larger than one.
Given two paths $\lambda\le\mu$, there exists at most one configuration of 
Dyck/ballot tiles in the region $R_{X}(\lambda,\mu)$ for a cover-exclusive 
tiling.

The following proposition is obvious from the definition of 
$\mathrm{wt}^{II}(\lambda,\mu)$. 

\begin{prop}
The matrix $N$ is expressed in terms of cover-exclusive Dyck tilings 
of type $D$:
\begin{align*}
N_{\lambda,\mu}=\tilde{P}_{\lambda,\mu}^{D}.
\end{align*}
\end{prop}

\begin{theorem}
An entry of the inverse of incidence matrices is expressed in terms of 
Dyck tilings of type $D$ as 
\begin{eqnarray*}
(M^{-1})_{\lambda,\mu}&=P^{D}_{\lambda,\mu}, \\
(N^{-1})_{\lambda,\mu}&=\widetilde{Q}^{D}_{\lambda,\mu}
\end{eqnarray*}
\end{theorem}
\begin{proof}
One can prove Theorem by the same argument in the proof of 
Theorem 6 in \cite{SZJ12}.
The differences are that a tile may have a two-by-two box, and 
the weight given to a tile depends on the choice of statistic on
the tile.

The weight $\mathrm{wt}^{I}(\lambda\leftarrow\mu)$ 
gives the statistics $\mathrm{art}$ on a tile, 
whereas $\mathrm{wt}^{II}(\lambda\leftarrow\mu)$
gives the statistics $\mathrm{tiles}$ on a tile.
\end{proof}

\begin{remark}
Some remarks are in order.
\begin{enumerate}
\item
When the length of a path is four with sign $\epsilon=0$, 
we have three entries of $M^{-1}$ (see Example \ref{ex:IM}) which 
are not a monomial of $q$. These correspond to non-trivial Dyck tilings of 
type $D$, that is, they contain a tile which is not a single box.
\item
An entry of the matrix $M$ is a generating function of cover-exclusive 
Dyck tilings of type $D$ with the weight $\mathrm{art}$.
\item
If we define the third weight $\mathrm{wt}^{III}(\lambda\leftarrow\mu)$ 
as the statistics $\mathrm{area}$ for the skew shape $R_{D}(\lambda,\mu)$, 
we obtain a generating function of cover-inclusive Dyck tilngs of type $D$ with 
statistics $\mathrm{area}$.
For example, when $\lambda=D^{2}U^{2}$ and $\mu=U^{4}$, we have two configurations 
of cover-inclusive Dyck tilings and their weights are both $q^{5}$. 
Thus, the generating function is $2q^{5}$.
\item 
The entries of the matrix $N$ are Kazhdan--Lusztig polynomials for the Hermite 
symmetric pair $(D_{n},A_{n-1})$ in \cite{LejStr13}.
In the language of \cite{SZJ12}, these Kazhdan--Lusztig polynomials $C_{w}^{-}$ 
are associated to the module $\mathcal{M}^{-}$.
Similarly, the matrix $N^{-1}$ gives also Kazhdan--Lusztig polynomials 
$C_{w}^{+}$ associated to the module $\mathcal{M}^{+}$.
\end{enumerate}
\end{remark}

\section{Enumeration of Dyck tilings of type \texorpdfstring{$D$}{D}}
\label{sec:EnuDTD}

\subsection{Enumeration of Dyck tilings of type \texorpdfstring{$A$}{A}}
Let $\lambda$ be a Dyck path of size $n$.
Recall that the generating function of Dyck tilings of type $A$ above $\lambda$ 
is defined by $P_{\lambda}^{A}$
and a Dyck tiling consists of only Dyck tiles.

To rewrite Eqn. (\ref{eqn:dtb}) for type $A$, we introduce a {\it chord} of a Dyck path.
We make a pair of $U$ and $D$ in $\lambda$ as ($\heartsuit 1$).
The pair is called a chord of the Dyck path.
We denote by $\mathrm{Chord}(\lambda)$ the set of chords of $\lambda$.
Note that we have $|\mathrm{Chord}(\lambda)|=n$ since we have $n$ pairs
of $U$ and $D$ in $\lambda$.
The {\it length} of a chord $c$ is one plus the number of chords in-between $U$ 
and $D$ in $c$.

The generating function (\ref{eqn:dtb}) for type $A$ can be expressed in terms 
of $q$-integers in a simple form:
\begin{theorem}[\cite{KW11,K12}]
The generating function of cover-inclusive Dyck tilings of type $A$ above 
$\lambda$ is given by 
\begin{align*}
P_{\lambda}^{A}=\frac{[n]!}{\prod_{c\in\mathrm{Chord}(\lambda)}[l(c)]}.
\end{align*}
\end{theorem}

\subsection{Enumeration of Dyck tilings of type \texorpdfstring{$B$}{B}}
\label{sec:enubt}
Let $\lambda$ be a ballot path.
Recall that a ballot tiling of type $B$ consists of Dyck tiles and 
ballot tiles.

Let $\lambda_{M,N}$ be a path $D^{M}U^{N}$.
We abbreviate the generating function of type $B$ for the path 
path $\lambda_{M,N}$ as $Q^{B}(M,N)$.

\begin{remark}
The generating function $Q^{B}(M,N)$ is studied in \cite[Section 7.1]{S17} 
as $P(M,N)$.
\end{remark}

We recall some results for $Q^{B}(M,N)$.
For positive integers $m$ and $N$, we define 
\begin{eqnarray*}
a_{(2m-1,N)}&:=&\frac{[N+2m]}{[2m]}, \\
a_{(2m,N)}&:=&\frac{[2N+2m]}{[N+2m]}.
\end{eqnarray*}
The function $Q^{B}(M,N)$ can be expressed in terms of $a_{(i,j)}$.
\begin{lemma}[{\cite[Proposition 7.2]{S17}}]
The generating function $Q^{B}(M,N)$ satisfies 
\begin{eqnarray*}
Q^{B}(M,N)=\prod_{1\le i\le N}(1+q^{i})\cdot \prod_{1\le j\le M}a_{(j,N)}.
\end{eqnarray*}
\end{lemma}

There are several expressions for a generating function of ballot tiles 
of type $B$ by using a tree structure associated to a ballot path $\lambda$.
The expressions are obtained by a bijection between a ballot tiling and 
a natural label on the tree (for example, see Section 7 and 8 in \cite{S17}). 
Especially, the generating function $P^{B}_{\lambda}$ has the following 
factorization.
\begin{theorem}[{\cite[Theorem 7.4]{S17}}]
Suppose that $\lambda$ is expressed as a concatenation of two paths 
$\lambda_1$ and $\lambda_2$ such that $\lambda=\lambda_1\circ\lambda_2$
and a Dyck path $\lambda_1$ cannot be written as a concatenation 
of two Dyck paths.
Let $M$ be the number of steps in $\lambda_2$ and $N$ be the length of 
the Dyck path $\lambda_1$.
Then, we have 
\begin{align*}
P^{B}_{\lambda}=P^{A}_{\lambda_1}\cdot Q^{B}(M,N)\cdot P^{B}_{\lambda_2}.
\end{align*}
\end{theorem}

\subsection{Enumeration of Dyck tilings of type \texorpdfstring{$D$}{D} with a lower 
fixed path}
In this subsection, we study the enumeration of Dyck tilings of type $D$ 
above a ballot path $\lambda$.
For simplicity, we first consider the generating functions for a path 
$\lambda_{M,N}:=D^{N}U^{M}$. 
Then, we apply the result to general cases.

We abbreviate the generating function $P^{D}_{\lambda_{M,N}}$ as $P(M,N)$.

We first show that the generating function $P(M,N)$ of type $D$ can be expressed 
in terms of $Q^{B}(M,N)$ defined in Section \ref{sec:enubt}.
\begin{prop}
\label{prop:Box}
Let $P(M,N)$ and $Q^{B}(M,N)$ be generating functions defined as above.
We have 
\begin{align*}
P(M,N)=Q^{B}(M-1,N).
\end{align*}
\end{prop}
\begin{proof}
From Eqn. (\ref{eqn:setpath}), it is easy to see the sign $\epsilon$ of the path 
$\lambda_{M,N}$ is given by $\epsilon\equiv N \pmod 2$.	

We show that Proposition is true for $\epsilon=0$. 
One can show the case of $\epsilon=1$ in a similar way.

Let $\epsilon=0$, or equivalently, $N$ even. 
Since there exists no partial path consisting of $UD$ in $\lambda_{M,N}$, 
we have no Dyck tiles of size larger than or equal to one.
In $\mathcal{D}(\lambda_{M,N}/\ast)$, we have configurations of Dyck tiles
which include a ballot tile of type $D$.
Namely, we have a ballot tile of type $D$ of length $(0,2p)$ with $p\ge1$.
Below, we show that there exists a bijection between a ballot tile of type 
$D$ and ballot tiles of type $B$, which preserves the statistics
$\mathrm{art}$.

Recall that a ballot tile of type $D$ is formed by two ballot tiles
and a single box. 
Here, the length of two ballot tiles is even.
On the other hand, in the case of type $B$, the length of 
ballot tiles of type $B$ is odd.

Suppose that a two-by-two tile forms a ballot tile of type $D$.
Then, we delete the two-by-two tile $t$ and attach a tile $t'$ at the 
position where the left-most one-by-one box of $t$ is placed.
We put asterisks at $t'$ and the box just below $t'$.
This operation yields a ballot tile of type $B$ from a ballot tile of 
type $D$. The discrepancy of the size of ballot tiles is resolved by the 
deletion of the two-by-two tile.

For example, we have a correspondence between ballot tiles of type $D$ and 
type $B$ in Figure \ref{fig:DB}.
\begin{figure}[ht]
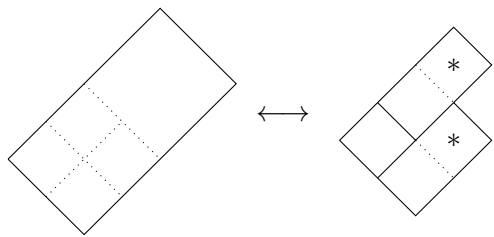

\tikzpic{-0.5}{
\draw(0,0)--(2,2)--(3,1)--(1,-1)--(0,0);
\draw[dotted](0.5,0.5)--(1.5,-0.5)(1,1)--(2,0)(0.5,-0.5)--(1.5,0.5);
}
$\longleftrightarrow$
\tikzpic{-0.5}{
\draw(0,0)--(1.5,1.5)--(2,1)--(1.5,0.5)--(2,0)--(1,-1)--(0,0);
\draw(0.5,0.5)--(1,0)(0.5,-0.5)--(1.5,0.5);
\draw[dotted](1,1)--(1.5,0.5)(1,0)--(1.5,-0.5);
\draw(1.5,1)node{$\ast$}(1.5,0)node{$\ast$};
}
\caption{A correspondence between a ballot tile of type $D$ and type $B$.}
\label{fig:DB}
\end{figure}
Further, it is an easy to show that a ballot tile of type $D$ and 
ballot tiles of type $B$ have the same statistics $\mathrm{art}$.
Actually, these two tiles have the statistics art three.

Other configurations in $\mathcal{D}(\lambda/\ast)$ have 
neither Dyck tiles nor ballot tiles.
Such a configuration consists of single boxes with $\mathrm{area}=1$ and 
$\mathrm{tiles}=1$ and two-by-two boxes with $\mathrm{area}=1$ and 
$\mathrm{tiles}=1$.
\begin{figure}[ht]
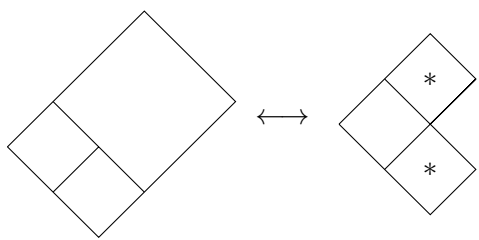

\tikzpic{-0.5}{[scale=0.6]
\draw(0,0)--(3,3)--(5,1)--(2,-2)--(0,0);
\draw(1,1)--(3,-1)(1,-1)--(2,0);
}
$\longleftrightarrow$
\tikzpic{-0.5}{[scale=0.6]
\draw(0,0)--(2,2)--(3,1)--(2,0)--(3,1)--(2,0)--(3,-1)--(2,-2)--(0,0);
\draw(1,1)--(2,0)--(1,-1);
\draw(2,1)node{$\ast$}(2,-1)node{$\ast$};
}
\caption{A correspondence between tilings consisting of single boxes and 
two-by-two boxes}
\label{fig:DB2}
\end{figure}
We delete a two-by-two box $t$ and put a box $a_{\ast}$ with $\ast$ where the left-most 
box of $t$ is placed. 
If there exists a box just below $a_{\ast}$, we also put $\ast$ on the box.
Then, we have a ballot tiling of type $B$.
See Figure \ref{fig:DB2} for an example of this replacement.
Note that the two tilings have the same $\mathrm{art}$ since 
the weight of a two-by-two is one which is the same as a single box.

As a consequence, we have an $\mathrm{art}$ preserving correspondence 
between a Dyck tiling of type $D$ and a ballot tiling of type $B$.
The difference of the sizes of two tilings is one, which 
implies $P(M,N)=Q^{B}(M-1,N)$.
\end{proof}

The correspondence between Dyck tilings of type $D$ and ballot tilings of type $B$
in the proof of Proposition \ref{prop:Box} gives stronger results on 
generating functions of Dyck tilings of type $D$.
Let $\lambda:=\lambda_1\lambda_2\ldots\lambda_{L}\in\{U,D\}^{L}$ be a path of type $D$.
We denote by $\widetilde{\lambda}$ a path $\lambda_{1}\ldots\lambda_{L-1}$.
Note that a non-trivial Dyck tile of type $D$ corresponds to two ballot 
tiles of type $B$ of the same shape and a single box attached to the two 
ballot tiles.
Then, by a similar argument to the proof of Proposition \ref{prop:Box},
we have the following theorem.
\begin{theorem}
\label{thrm:DB}
The generating function $P_{\lambda}$ can be expressed in terms of 
the generating function $P^{B}_{\widetilde{\lambda}}$:
\begin{eqnarray*}
P^{D}_{\lambda}=P^{B}_{\widetilde{\lambda}}.
\end{eqnarray*}
\label{thrm:DTDbT1}
\end{theorem}

Let $\lambda_{1}$ is a Dyck path of length $N$ which is not written as a concatenation 
of two Dyck paths, and $\lambda_{2}$ be a path of type $D$.
Let $M$ be the number of steps in $\lambda_2$.
We denote the concatenation of $\lambda_{1}$ and $\lambda_{2}$
by $\lambda:=\lambda_1\circ \lambda_{2}$.

\begin{cor}
Let $\lambda_1$, $\lambda_2$, and $\lambda:=\lambda_{1}\circ\lambda_{2}$ be paths as above.
The generating function $P_{\lambda}$ has the following factorization:
\begin{align*}
P^{D}_{\lambda}=P^{A}_{\lambda_{1}}\cdot P(M,N)\cdot P^{D}_{\lambda_2}.
\end{align*}
\end{cor}
\begin{proof}
Theorem \ref{thrm:DB} and the factorization of a generating function
of ballot tilings of type $B$ (see Theorem 7.4 in \cite{S17}) 
implies the factorization of a generating function of Dyck tilings of type $D$.
\end{proof}

\subsection{Link patterns and trees}
Let $w$ be a sequence of $U$ and $D$.
Then, we construct a {\it link pattern} from $w$.
We first perform the operations from $(\heartsuit1)$ to $(\heartsuit3)$ 
as in Section \ref{sec:IM}.
Then, we perform operations successively as follows.
\begin{enumerate}
\item[($\heartsuit4$)] 
Suppose that we have $N_{D}$ unpaired $D$'s and $N_{U}$ unpaired $U$'s 
after the operation $(\heartsuit3)$. Note that $N_{U}$ is at most one. 
Then, we add $N_{D}+N_{U}$ $U$'s at the left end of $w$.

\item[($\heartsuit5$)] We perform the operations from $(\heartsuit1)$ 
to $(\heartsuit3)$ on the new $w$.

\item[($\heartsuit6$)] If the character of the right end of $w$ is $D$, then 
we change the arc formed by this $D$ into a dashed arc.
\end{enumerate}
An {\it outer arc}  is an arc which has no arcs of larger size 
above it.
By definition of dashed arcs, dashed arcs are always outer arcs.

\begin{example}
Let $w:=DDUDUUD$. 
Then, the link pattern associated with $w$ is  
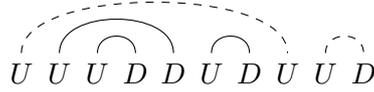
\begin{eqnarray*}
\tikzpic{-0.5}{
\draw(0,0)node[anchor=north]{$D$}(0.5,0)node[anchor=north]{$D$}
(1,0)node[anchor=north]{$U$}(1.5,0)node[anchor=north]{$D$}
(2,0)node[anchor=north]{$U$}(2.5,0)node[anchor=north]{$U$}
(3,0)node[anchor=north]{$D$}
(-0.5,0)node[anchor=north]{$U$}(-1,0)node[anchor=north]{$U$}
(-1.5,0)node[anchor=north]{$U$};
\draw(-0.5,0)..controls(-0.5,0.3)and(0,0.3)..(0,0);
\draw(-1,0)..controls(-1,0.6)and(0.5,0.6)..(0.5,0);
\draw(1,0)..controls(1,0.3)and(1.5,0.3)..(1.5,0);
\draw[dashed](2.5,0)..controls(2.5,0.3)and(3,0.3)..(3,0);
\draw[dashed](-1.5,0)..controls(-1.5,0.9)and(2,0.9)..(2,0);
}
\end{eqnarray*}
The two dashed arcs are outer arcs.
\end{example}

We construct a tree from a path $\lambda$.
Let $\mathcal{Z}$ be the set of Dyck words.

Let $\lambda$ be a sequence of $U$ and $D$ and 
$\pi(\lambda)$ be a link pattern associated with $\lambda$.
We define a plane tree $A(\lambda)$ associated with $\lambda$.
We append $U^{p}$ with some $p\ge1$ to the sequence $\lambda$ if necessary.
A tree $A(\lambda)$ is recursively obtained as follows.
\begin{enumerate}[($\diamondsuit$1)]
\item $A(\emptyset)$ is the empty tree.
\item $A(z\lambda)$, $z\in \mathcal{Z}$, is obtained by attaching the trees
$A(z)$ and $A(\lambda)$ at their roots.
\item $A(UzD)$ is obtained by attaching an edge 
above the root of $A(z)$.
If $U$ and $D$ in $UzD$ form a dashed arc, we put a dot at the edge 
corresponding to this arc.
\item  $A(UzU\lambda)$ is obtained by attaching an edge with a dot above the
root of the tree $A(z\lambda)$.
\end{enumerate}
We encode an additional information on a tree $A(\lambda)$.
Fix a dashed arc $a_{0}$ in $\pi(\lambda)$.
We enumerate outer arcs (without dashes) which is left to $a_0$ and right to the 
rightmost dashed arc left to $a_{0}$ by $a_{1}, a_{2},\ldots, a_{m}$ 
from right to left.
If there are no dashed arcs left to $a_{0}$, we enumerate outer arcs left to 
$a_{0}$ by $a_1,\ldots,a_{m}$ from right to left.
From $(\diamondsuit2)$ to $(\diamondsuit4)$, each arc has a corresponding 
edge in $A(\lambda)$.
We put an arrow from the edge for $a_{i-1}$ to the edge for $a_{i}$ for 
$1\le i\le m$.

\begin{example}
\label{ex:tree1}
Let $\lambda=DUUDUU$. Then, the link pattern and the plane tree are as follows.
\begin{eqnarray*}
\tikzpic{-0.5}{
\draw(0,0)node[anchor=north]{$D$}(0.5,0)node[anchor=north]{$U$}
(1,0)node[anchor=north]{$U$}(1.5,0)node[anchor=north]{$D$}
(2,0)node[anchor=north]{$U$}(2.5,0)node[anchor=north]{$U$}
(-0.5,0)node[anchor=north]{$U$}(-1,0)node[anchor=north]{$U$};
\draw(-0.5,0)..controls(-0.5,0.3)and(0,0.3)..(0,0);
\draw[dashed](-1,0)..controls(-1,0.6)and(0.5,0.6)..(0.5,0);
\draw(1,0)..controls(1,0.3)and(1.5,0.3)..(1.5,0);
\draw[dashed](2,0)..controls(2,0.3)and(2.5,0.3)..(2.5,0);
} \qquad
\scalebox{0.7}{
\tikzpic{-0.5}{
\coordinate 
	child{coordinate (c0)
		child
		child{coordinate (c1)}
		child{coordinate (c2)}
	};
\draw($(0,0)!.5!(c0)$)node{\scalebox{1.5}{$\bullet$}};
\node at ($(c0)!.5!(c2)$){\scalebox{1.5}{$\bullet$}};
\draw[latex-,dashed,very thick]($(c0)!.7!(c1)$)--($(c0)!.7!(c2)$);
}
}
\end{eqnarray*}
\end{example}

\begin{remark}
From Theorem \ref{thrm:DB}, the generating function $P^{D}_{\lambda}$ is equal 
to the generating function $P^{B}_{\widetilde{\lambda}}$.
Recall that we obtain $\widetilde{\lambda}$ by deleting the last step 
in the path $\lambda$.
This operation corresponds to the operation $(\heartsuit6)$.
To see this, we consider the case of type $B$.
Note that a link pattern is dual to a plane tree.
Here, ``dual" means that there is a one-to-one correspondence between 
a link in a link pattern and an edge of a tree.
The construction of a plane tree of type $B$ is studied in \cite{Boe88} and 
that of a link pattern is studied as type BI Kazhdan--Lusztig bases 
in \cite{S141}. 
The construction of a link pattern of type $B$ is the same as type $D$ 
from $(\heartsuit1)$ to $(\heartsuit3)$.
Then, the rightmost $U$ is isolated and marked.
The remaining $U$'s are paired as in $(\heartsuit3)$.
It is easy to see that the edge corresponding to the isolated 
$U$ in a link pattern plays the same role as a dashed arc.
Thus, when we delete the last step from $\lambda$, by taking into account
the operation $(\heartsuit6)$, the tree $A(\lambda)$ of type $D$ is 
the same as the tree for $\widetilde{\lambda}$ of type $B$.
\end{remark}

The generating function can be computed from a tree $A(\lambda)$ by the
following transformation of the tree. 
The algorithm for a transformation is the same as in the case of ballot tilings
considered in \cite[Section 7.3]{S17}.
For this paper to be self-contained, we list up transformations from \cite[Section 7.3]{S17}.
\begin{eqnarray*}
\tikzpic{-0.5}{
\draw(0,0)--(-0.3,-0.3)(-0.7,-0.7)--(-1,-1);
\draw(-0.2,-0.2)node{\rotatebox{-45}{$-$}}(-0.8,-0.8)node{\rotatebox{-45}{$-$}};
\draw[dashed](-0.3,-0.3)--(-0.7,-0.7);
\draw[decoration={brace,mirror,raise=5pt},decorate]
  (0,0) --(-1,-1);
\draw(-0.2,-0.2)node[left=9pt]{$N$};
\draw(0,0)--(0.3,-0.3)(0.7,-0.7)--(1,-1);
\draw[dashed](0.3,-0.3)--(0.7,-0.7);
\draw(0.2,-0.2)node{\rotatebox{45}{$-$}}(0.8,-0.8)node{\rotatebox{45}{$-$}};
\draw[decoration={brace,mirror,raise=5pt},decorate]
  (1,-1)--(0,0);
\draw(0.2,-0.2)node[right=9pt]{$M$};
}
&\mapsto&\genfrac{[}{]}{0pt}{}{M+N}{M}\cdot
\tikzpic{-0.4}{
\draw(0,0)--(0,-0.3)(0,-0.7)--(0,-1);
\draw(0,-0.2)node{$-$}(0,-0.8)node{$-$};
\draw[dashed](0,-0.3)--(0,-0.7);
\draw[decoration={brace,mirror,raise=5pt},decorate]
    (0,-1)-- (0,0);
\draw(0,-0.5)node[right=9pt]{$M+N$};
} \\
\tikzpic{-0.5}{
\draw(0,0)--(-0.3,-0.3)(-0.7,-0.7)--(-1,-1);
\draw(-0.2,-0.2)node{\rotatebox{-45}{$-$}}(-0.8,-0.8)node{\rotatebox{-45}{$-$}};
\draw[dashed](-0.3,-0.3)--(-0.7,-0.7);
\draw[decoration={brace,mirror,raise=5pt},decorate]
  (0,0) --(-1,-1);
\draw(-0.2,-0.2)node[left=9pt]{$N$};
\draw(0,0)--(0.35,-0.35)(0.65,-0.65)--(1,-1);
\draw[dashed](0.35,-0.35)--(0.65,-0.65);
\draw(0.3,-0.3)node{\rotatebox{45}{$-$}}(0.7,-0.7)node{\rotatebox{45}{$-$}};
\draw(0.15,-0.15)node{$\bullet$}(0.85,-0.85)node{$\bullet$};
\draw[decoration={brace,mirror,raise=5pt},decorate]
  (1,-1)--(0,0);
\draw(0.2,-0.2)node[right=9pt]{$M$};
}&\mapsto&
\genfrac{[}{]}{0pt}{}{M+N}{M}_{q^{2}}\prod_{i=1}^{N}(1+q^{i})\cdot
\tikzpic{-0.4}{
\draw(0,0)--(0,-0.35)(0,-0.65)--(0,-1);
\draw(0,-0.3)node{$-$}(0,-0.7)node{$-$};
\draw[dashed](0,-0.4)--(0,-0.6);
\draw(0,-0.15)node{$\bullet$}(0,-0.85)node{$\bullet$};
\draw[decoration={brace,mirror,raise=5pt},decorate]
    (0,-1)-- (0,0);
\draw(0,-0.5)node[right=9pt]{$M+N$};
} \\
\tikzpic{-0.5}{
\draw(0,0)--(-0.45,-0.45)(-0.7,-0.7)--(-1,-1);
\draw(-0.4,-0.4)node{\rotatebox{-45}{$-$}}(-0.8,-0.8)node{\rotatebox{-45}{$-$}};
\draw[dashed](-0.3,-0.3)--(-0.7,-0.7);
\draw[decoration={brace,mirror,raise=5pt},decorate]
  (0,0) --(-1,-1);
\draw(-0.2,-0.2)node[left=9pt]{$N$};
\draw(0,0)--(0.45,-0.45)(0.65,-0.65)--(1,-1);
\draw[dashed](0.45,-0.45)--(0.65,-0.65);
\draw(0.4,-0.4)node{\rotatebox{45}{$-$}}(0.7,-0.7)node{\rotatebox{45}{$-$}};
\draw(0.15,-0.15)node{$\bullet$}(0.85,-0.85)node{$\bullet$};
\draw[decoration={brace,mirror,raise=5pt},decorate]
  (1,-1)--(0,0);
\draw(0.2,-0.2)node[right=9pt]{$M$};
\draw[latex-,dashed](-0.3,-0.3)--(0.3,-0.3);
}&\mapsto&
\genfrac{[}{]}{0pt}{}{M+N}{M}_{q^{2}}
\frac{[2M+N]}{[2(M+N)]}\prod_{i=1}^{N}(1+q^{i}) \cdot
\tikzpic{-0.4}{
\draw(0,0)--(0,-0.35)(0,-0.65)--(0,-1);
\draw(0,-0.3)node{$-$}(0,-0.7)node{$-$};
\draw[dashed](0,-0.4)--(0,-0.6);
\draw(0,-0.15)node{$\bullet$}(0,-0.85)node{$\bullet$};
\draw[decoration={brace,mirror,raise=5pt},decorate]
    (0,-1)-- (0,0);
\draw(0,-0.5)node[right=9pt]{$M+N$};
\draw[latex-,dashed](-0.4,-0.15)node{$($}--(-0.1,-0.15)node{$)$};
}
\end{eqnarray*}
Here $(\leftarrow)$ in the right hand side of the third operation means that 
if the leftmost top edge in the left hand side of the third operation has an outgoing arrow,
we put an outgoing arrow on the top edge of the right hand side.
Note that the transformation is also valid when $M=0$.

Then, we define operations on the following trees (not a partial tree):
\begin{eqnarray*}
\tikzpic{-0.4}{
\draw(0,0)--(0,-0.3)(0,-0.7)--(0,-1);
\draw(0,-0.2)node{$-$}(0,-0.8)node{$-$};
\draw[dashed](0,-0.3)--(0,-0.7);
\draw[decoration={brace,mirror,raise=5pt},decorate]
    (0,-1)-- (0,0);
\draw(0,-0.5)node[right=9pt]{$N$};
}&\mapsto&\prod_{i=1}^{N}(1+q^{i}), \\
\tikzpic{-0.4}{
\draw(0,0)--(0,-0.35)(0,-0.65)--(0,-1);
\draw(0,-0.3)node{$-$}(0,-0.7)node{$-$};
\draw[dashed](0,-0.4)--(0,-0.6);
\draw(0,-0.15)node{$\bullet$}(0,-0.85)node{$\bullet$};
\draw[decoration={brace,mirror,raise=5pt},decorate]
    (0,-1)-- (0,0);
\draw(0,-0.5)node[right=9pt]{$N$};
}&\mapsto& 1.
\end{eqnarray*}
Then, we have a map from a tree $A(\lambda)$ to $\mathbb{Z}[q]$ by 
successive applications of the operations defined above.
We denote this map by $\omega:A(\lambda)\mapsto\mathbb{Z}[q]$.

From Theorem \ref{thrm:DB} and factorization of 
the generating function of type $B$ (Theorem 7.4 in \cite{S17}),
the generating function $P_{\lambda}$ can be expressed in terms 
of a product of $q$-integers.
\begin{theorem}
Let $\lambda$ be a path of type $D$, and $A(\lambda)$ be the 
tree constructed from $\lambda$.
Then, we have 
\begin{eqnarray*}
P^{D}_{\lambda}=\omega(A(\lambda)).
\end{eqnarray*}
\end{theorem}

\begin{example}
We consider the same tree as in Example \ref{ex:tree1}.
The action of $\omega$ on the tree is given as follows.
\begin{eqnarray*}
\scalebox{0.7}{
\tikzpic{-0.5}{
\coordinate 
	child{coordinate (c0)
		child
		child{coordinate (c1)}
		child{coordinate (c2)}
	};
\draw($(0,0)!.5!(c0)$)node{\scalebox{1.5}{$\bullet$}};
\node at ($(c0)!.5!(c2)$){\scalebox{1.5}{$\bullet$}};
\draw[latex-,dashed,very thick]($(c0)!.7!(c1)$)--($(c0)!.7!(c2)$);
}
}
=
\frac{[4]}{[2]}\frac{[3]}{[4]}[2]
\scalebox{0.6}{
\tikzpic{-0.5}{
\coordinate 
	child{coordinate (c0)
		child
		child[missing]
		child{coordinate (c2)
			child[missing]
			child[missing]
			child{coordinate (c3)}		
		}
	};
\draw($(0,0)!.5!(c0)$)node{\scalebox{1.5}{$\bullet$}};
\node at ($(c0)!.5!(c2)$){\scalebox{1.5}{$\bullet$}};
\node at ($(c3)!.5!(c2)$){\scalebox{1.5}{$\bullet$}};
\draw(c2)node{\rotatebox{45}{\scalebox{1.5}{$-$}}};
}
}=[3][6]
\end{eqnarray*}
\end{example}

\subsection{Generating functions with a fixed upper path}
By a similar argument to Proposition \ref{prop:Box} and Theorem \ref{thrm:DB}, 
it is easy to see that the correspondence between a Dyck tiling of type $D$ and 
a ballot tiling preserves also the statistics $\mathrm{tiles}$.
Thus, the generating function $\widetilde{P}_{\mu}$ can be 
expressed in terms of the generating function of type $B$.

Recall that given a path $\mu:=\mu_{1}\ldots\mu_{L}$ of type $D$, 
we define $\tilde{\mu}:=\mu_{1}\ldots\mu_{L-1}$.
Then, we have 
\begin{theorem}
\label{thrm:DB2}
The generating function $\widetilde{P}_{\mu}$ is expressed as
\begin{eqnarray*}
\widetilde{P}^{D}_{\mu}=\widetilde{P}^{B}_{\tilde{\mu}}.
\end{eqnarray*}
\end{theorem}

\bibliographystyle{amsplainhyper} 
\bibliography{biblio}

\end{document}